\documentclass{birkjour}
\usepackage[all]{xy}
\usepackage{hyperref}  
\usepackage{cleveref}
   \usepackage{latexsym} 
  \usepackage[all]{xy}
  \usepackage{amsfonts} 
  \usepackage{amsthm} 
  \usepackage{amsmath} 
  \usepackage{amssymb}
  \usepackage{pifont}  
  \usepackage{enumerate}
  \usepackage{dcolumn}
  \usepackage{comment}

\newcolumntype{2}{D{.}{}{2.0}}
  \xyoption{2cell}

 \usepackage{tikz}
\usetikzlibrary{arrows}

  \def\sw#1{{\sb{(#1)}}}

   \def\<{{\langle}} 
  \def\>{{\rangle}} 
   
  \def\la{{\triangleright}} 
   
  \def\span{{\rm span}}

  \def\note#1{{}} 
 \def\can{{\rm \textsf{can}}}

  \def\note#1{}

  \def\cE{{\mathcal E}}

  \def\cO{{\mathcal O}}

  \def\beq{\begin{equation}} 
  \def\eeq{\end{equation}}

  \def\id{\mathrm{id}}

  \def\ot{{\otimes}}

  \newcounter{zlist} 
  \newenvironment{zlist}{\begin{list}{(\arabic{zlist})}{ 
  \usecounter{zlist}\leftmargin2.5em\labelwidth2em\labelsep0.5em 
  \topsep0.6ex 
  \parsep0.3ex plus0.2ex minus0.1ex}}{\end{list}}

  \newcounter{blist}

  \newcounter{rlist}

\def\stac#1{\raise-.2cm\hbox{$\stackrel{\displaystyle\otimes}{\scriptscriptstyle{#1}}$}}

\def\cten#1{\raise-.2cm\hbox{$\stackrel{\displaystyle\reallywidehat{\otimes}}
{\scriptscriptstyle{#1}}$}}

  \def\Label#1{\label{#1}\ifmmode\llap{[#1] }\else 
  \marginpar{\smash{\hbox{\tiny [#1]}}}\fi} 
  \def\Label{\label}

  \newtheorem{proposition}{Proposition}[section]
  \newtheorem{lemma}[proposition]{Lemma} 
  \newtheorem{corollary}[proposition]{Corollary}

\theoremstyle{definition}

  \theoremstyle{remark}

\renewcommand{\theequation}{\thesection.\arabic{equation}} 
  \newcounter{c} 
   
  \newcommand{\etyk}[1]{\vspace{-7.4mm}$$\begin{equation}\Label{#1} 
  \addtocounter{c}{1}} 
  \renewcommand{\]}{\ifnum \value{c}=1 $$\else \end{equation}\fi} 
   \numberwithin{equation}{section}

\def\ot{\otimes}

\def\AA{{\mathbb A}}

\def\CC{{\mathbb C}}

\def\NN{{\mathbb N}}

\def\ZZ{{\mathbb Z}}

\def\tt{{\mathfrak T}}

\newcommand{\Cc}{\mathcal{C}}

\def\O{{\bf O}}

\def\*C{{}^*\hspace*{-1pt}{\Cc}}

\def\text#1{{\rm {\rm #1}}}

 \def\1{\mathbf{1}}

\def\bpartial{\bar{\partial}}

\def\id{\mathrm{id}}

  \def\sw#1{{\sb{(#1)}}}

\def\1\mathbf{1}

\def\|#1{\overline{#1}}

\newcommand\reallywidehat[1]{%
\savestack{\tmpbox}{\stretchto{%
  \scaleto{%
    \scalerel*[\widthof{\ensuremath{#1}}]{\kern.1pt\mathchar"0362\kern.1pt}%
    {\rule{0ex}{\textheight}}
  }{\textheight}%
}{2.4ex}}%
\stackon[-6.9pt]{#1}{\tmpbox}%
}
\begin{document}

\title[Differential calculus on the non-standard quantum sphere]{Remarks on covariant differential calculus on the non-standard quantum Podle\'s sphere}

\author[Brzezi\'nski]{Tomasz Brzezi\'nski}

\address{
Department of Mathematics, Swansea University,
Fabian Way,
  Swansea SA1 8EN, U.K.\ \newline \noindent
Faculty of Mathematics, University of Bia{\l}ystok, K.\ Cio{\l}kowskiego  1M,
15-245 Bia\-{\l}ys\-tok, Poland}

\email{T.Brzezinski@swansea.ac.uk, T.Brzezinski@uwb.edu.pl}

\author[Ó Buachalla]{R\'eamonn Ó Buachalla}

\address{Mathematical Institute of Charles University, Sokolovsk\'a 83, Prague, Czech Republic}

\email{reamonnobuachalla@gmail.com}


\thanks{This research is supported through WEAVE-UNISONO by the National Science Centre, Poland grant no.\ 2023/05/Y/ST1/00046, and by GA\v{C}R grant no. CA24-11728K. R\'OB acknowledges support from COST Action 21109
CaLISTA, supported by COST (European Cooperation in Science and Technology) and HORIZON-MSCA-2022-SE-01-01 CaLIGOLA}
 
\begin{abstract}
An $SU_q(2)$-covariant first-order differential calculus on the non-standard quantum Podle\'s two-sphere is constructed. The module of one-forms is explicitly given in terms of modules of sections of charge $\pm 2$ line bundles associated to the quantum Hopf fibration. 
\end{abstract}    
\date\today
\maketitle

\section{Introduction}
The quantum Podle\'s 2-sphere \cite{Pod:sph} is one of the most fundamental and most studied examples of a quantum or non-commutative space. It is often taken as an illustration of both the richness and the non-trivial nature of noncommutative geometry arising from quantum groups. It is a quantum homogeneous space and a base for the noncommutative counterpart of the classical Hopf fibration. Going beyond the purely algebraic investigations, the quantum Podle\'s 2-sphere is an example of a graph $C^*$-algebra \cite{HonSzy:sph} and a highly nontrivial example of a spectral noncommutative geometry \`a la Connes \cite{DabSit:Dir}, \cite{DabLan:spe}. 

The algebra of coordinate functions on the quantum sphere,  $\cO_{q,s}(S^2) $, is a two-parameter deformation of the algebra of functions on the two-sphere. To maintain the connection with the $C^*$-algebra version of the quantum two-sphere, the parameters are typically taken from the intervals $q\in (0,1)$, $s\in [0,1]$. The limiting case $s=0$, $q\to 1$ corresponds to the commutative algebra of functions on the two-sphere. In general, the case $s=0$, referred to as the \emph{standard quantum sphere}, and denoted by $\cO(S_q^2)$, differs significantly from the case $s\neq 0$, known as the \emph{non-standard quantum sphere}. The standard quantum sphere $\cO(S_q^2)$ arises as the quotient of the quantum group $SU_q(2)$ by its quantum subgroup $U(1)$. It can be and is identified with the zero-degree subalgebra of the natural $\ZZ$-grading of the algebra of functions on $SU_q(2)$. Furthermore, this grading is strong (meaning that every element of $\cO(S_q^2)$ is a sum of products of elements of opposite degrees). As $\cO(S_q^2)$ is the fixed point or coinvariant subalgebra of the coaction of $\cO(U(1))$ on $\cO_{q}(SU(2)) $ induced by the Hopf-algebra surjection $\cO_{q}(SU(2)) \to \cO(U(1))$, many algebraic properties of $\cO_{q}(SU(2)) $ transfer to $\cO(S_q^2)$. From the perspective of this note, the fact that the left covariant first-order differential calculus on $\cO_{q}(SU(2)) $ introduced in \cite{Wor:dif}, induces a two-dimensional covariant differential calculus on $\cO(S_q^2)$, first analyzed in \cite{Pod:dif}, \cite{Pod:cla} is of key significance.  
By \cite{Pod:cla}, there is a unique covariant two-dimensional differential calculus on $\cO(S_q^2)$ with module of one-forms generated (as a left $\cO(S_q^2)$-module) by one-forms $d e_i$, where $e_i$, $i=0,\pm$ are generators of $\cO(S_q^2)$. As realized in \cite{Maj:Rie}, the module of one-forms splits into the direct sum of (finitely generated projective) modules of sections of line bundles of monopole charge $\pm 2$ associated in \cite{HajMaj:pro} with the quantum Hopf fibration constructed in \cite{BrzMaj:gau}. These direct summands have a natural interpretation as holomorphic and anti-holomorphic forms \cite{KhaLan:hol}, \cite{OBu:com}. This parallels exactly the classical geometry of the two-sphere.

If $s\neq 0$ the situation is entirely different. The coalgebra surjection $\cO_{q}(SU(2)) \to \cO(U(1))$ that generates the coaction with coinvariants $\cO_{q,s}(S^2) $ is no longer a Hopf-algebra map.  The induction process that produces a 2-dimensional calculus from the $3$-dimensional calculus on $\cO_{q}(SU(2)) $ can no longer be used. As observed  in  \cite{Pod:cla}, there are no $2$-dimensional covariant calculi on $\cO_{q,s}(S^2) $ generated as a left $\cO_{q,s}(S^2) $-module by the $de_i$. Even more remarkably, Heckenberger and Kolb have proven in \cite{HecKol:Pod} that all covariant differential calculi on $\cO_{q,s}(S^2) $ that are generated as a left module by the $de_i$ are free as left and right modules. This is counterintuitive, as it would correspond to the statement that the cotangent bundle over the sphere is trivial.

The aim of this note is to demonstrate that there are covariant first-order differential calculi over $\cO_{q,s}(S^2) $, which split into holomorphic and anti-holomorphic parts given by modules closely related to modules of sections of line bundles of charges $\pm 2$, and thus reflect closely both the classical and the standard quantum sphere cases. 

\section{Notation and conventions}\label{sec.prelim}
In this section we briefly recall algebras that feature in this note. We assume that the reader  is familiar with the theory of quantum groups. The main reference is \cite{KliSch:gro}.

The coordinate algebra of the quantum group $SL(2)$, $\cO_q(SL(2))$, is a complex matrix quantum group generated by $\alpha,\beta,\gamma,\delta$ arranged in a $2\times 2$ matrix
$\begin{pmatrix}
    \alpha&\beta\cr\gamma&\delta
\end{pmatrix}$
and relations
$$
\alpha\beta=q\beta\alpha,\quad  \alpha\gamma =q\gamma\alpha, \quad \alpha\delta
= \delta\alpha + (q-q^{-1})\beta\gamma,
$$
$$\beta\gamma = \gamma\beta,\quad \beta\delta = q\delta\beta, \quad
\gamma\delta = q\delta\gamma, \quad \alpha\delta - q\beta\gamma =1,
$$
where $q$ is a complex parameter.
The comultiplication $\Delta$ and counit $\varepsilon$ are of matrix type, that is,
$$
\Delta\begin{pmatrix}
    \alpha&\beta\cr\gamma&\delta
\end{pmatrix} = \begin{pmatrix}
    \alpha&\beta\cr\gamma&\delta
\end{pmatrix}\otimes \begin{pmatrix}
    \alpha&\beta\cr\gamma&\delta
\end{pmatrix}, \qquad \varepsilon\begin{pmatrix}
    \alpha&\beta\cr\gamma&\delta
\end{pmatrix}=\begin{pmatrix}
    1 &0\cr0&1
\end{pmatrix}.
$$
The antipode $S$ is the inverse of $\begin{pmatrix}
    \alpha&\beta\cr\gamma&\delta
\end{pmatrix}$, that is:
$$
S\alpha =\delta, \quad S\beta = -q^{-1}\beta, \quad S\gamma = -q\gamma, \quad S\delta = \alpha.
$$

If $q\in \mathbb{R}$, then $\cO_q(SL(2))$ admits a $*$-algebra structure with the involution defined by
$$
\delta = \alpha^*, \quad \beta = - q\gamma^*.
$$
$\cO_q(SL(2))$ with this $*$-algebra structure is the algebra of coordinate functions on the quantum group $SU(2)$, which is denoted by $\cO_{q}(SU(2)) $. If, furthermore, $q\in (0,1)$, then $\cO_{q}(SU(2)) $ can be completed universally to the $C^*$-algebra of continuous functions on the quantum group $SU(2)$.

The algebra $\cO_{q,s}(S^2) $ of the coordinate functions on the quantum Podle\'s sphere is defined as a left coideal subalgebra of $\cO_q(SL(2))$ generated by 
$$
\xi = s(q^{-1}\beta^2-\alpha^2 ) +(1-s^2)q^{-1}\alpha\beta, \quad \eta =
s(q\gamma^2-\delta^2)
+(s^2-1)\gamma\delta,
$$
$$ \zeta = s(\beta\delta- q\alpha\gamma)
+ (1-s^2)q\beta\gamma ,
$$
where $s\in [0,1]$. The left coideal subalgebra property means that $\Delta(\cO_{q,s}(S^2) )\subset \cO_q(SL(2))\otimes \cO_{q,s}(S^2) $. The derived algebra relations are:
\begin{equation}\label{sphere}
\begin{aligned}
     \xi\zeta = q^{2}\zeta\xi, \quad &\eta\zeta = q^{-2}\zeta\eta, \quad \xi\eta = (s^2-\zeta)(\zeta + 1),\\  &\eta\xi = (s^2-q^{-2}\zeta)(q^{-2}\zeta + 1).
\end{aligned}
   \end{equation}
Thus abstractly, $\cO_{q,s}(S^2) $ is an algebra generated by $\xi,\eta,\zeta$ and relations \cref{sphere}.

If $q\in \mathbb{R}$, the $*$-structure of $\cO_{q}(SU(2)) $ restricts to $\cO_{q,s}(S^2) $ thus making it into a $*$-algebra with  $\zeta^*=\zeta$ and $\eta = \xi^*$. If, further, $q\in (0,1)$, the $C^*$-completion yields the algebra of continuous functions on the quantum sphere originally introduced in \cite{Pod:sph}.

The quantised universal enveloping algebra $U_q(sl_2)$ is generated  by $K^{\pm 1}$, $E$, $F$ with relations:
$$
KE=q^2 EK, \quad KF=q^{-2}FK,\quad  [E,F] = \frac{K-K^{-1}}{q-q^{-1}}.
$$
This is a Hopf algebra with comultiplication $\Delta$, counit $\varepsilon$ and the antipode $S$:
$$
\begin{aligned}
    &\Delta(K^{\pm 1}) = K^{\pm 1}\otimes K^{\pm 1}, \qquad  \Delta (E) =E\otimes K + 1\otimes E, \\
    &\Delta (F) =F\otimes 1 + K^{-1}\otimes F,\\
    & \varepsilon(K^{\pm 1}) = 1, \qquad \varepsilon(E) = \varepsilon(F)=0,\\
    &S(K^\pm 1) = K^{\mp 1}, \qquad S(E) = -EK^{-1},\qquad S(F) = -KF.
\end{aligned}
$$

$U_q(sl_2)$ and $\cO_q(SL(2))$ are dually paired Hopf algebras. On generators the Hopf algebra pairing 
$\langle-,-\rangle: U_q(sl_2)\times \cO_q(SL(2))\to \CC$ is given by:
\begin{equation}\label{pairing}
    \langle K,\alpha\rangle = q^{-1}, \quad \langle K,\delta\rangle = q, \quad \langle E,\gamma\rangle = \langle F,\beta\rangle = 1;
\end{equation}
see \cite[p.\ 120, Theorem~21]{KliSch:gro}.
As a result of this, the left Hopf action of $U_q(sl_2)$ on $\cO_q(SL(2))$, which in general is given by the formula
\begin{equation}\label{act.gen}
   X\la x = \sum x\sw 1\langle X,x\sw 2\rangle,\qquad \mbox{for all $X\in U_q(sl_2)$,\, $x\in \cO_q(SL(2))$,}
\end{equation}
where $\sum x\sw 1\otimes x\sw 2 = \Delta(x)$ is the Sweedler notation for the coproduct,
comes out as
\begin{subequations}\label{action}
    \begin{equation}\label{act.e}
        E\la \alpha = \beta, \quad E\la \beta = 0, \quad E\la \gamma = \delta, \quad E\la\delta = 0,
        \end{equation}
        \begin{equation}\label{act.k}
            K\la \alpha = q^{-1}\alpha, \quad K\la \beta = q\beta, \quad K\la \gamma = q^{-1}\gamma , \quad K\la\delta = q\delta,
        \end{equation}
\begin{equation}\label{act.f}
        F\la \alpha = 0, \quad F\la \beta = \alpha, \quad F\la \gamma = 0, \quad F\la\delta = \gamma.
        \end{equation}
\end{subequations}

If $q\in \mathbb{R}$, the real form $U_q(su_2)$  of $U_q(sl_2)$ is obtained by making $U_q(sl_2)$ into a $*$-Hopf algebra by setting
\begin{equation}\label{U*}
    K^* = K, \qquad E^* = KF.
\end{equation}
In the case of the real forms, the pairing \eqref{pairing} is compatible with the $*$-structure in the sense that, 
\begin{equation}\label{*comp}
 \langle K,x^*\rangle = \overline{\langle K^{-1},x\rangle}, \quad   \langle E,x^*\rangle = -q\overline{\langle F,x\rangle}, \quad \langle F,x^*\rangle = -q^{-1}\overline{\langle E,x\rangle},
\end{equation}
for all $x\in \cO_{q}(SU(2)) $; see \cite[Section~4.4.1]{KliSch:gro}. 

From now on we assume that $q\in \mathbb{R}$ and we consider the above algebras as $*$-algebras with involutions just described.

\section{First-order covariant differential calculus arising from  line bundles}\label{sec.diff}
The aim of this section is to construct a first-order covariant differential calculus on $\cO_{q,s}(S^2) $ that arises from charge $\pm 2$ line bundles over $\cO_{q,s}(S^2) $. 

By a covariant first-order differential calculus on $\cO_{q,s}(S^2) $ we mean an $\cO_{q,s}(S^2) $-bimodule and a left $\cO_{q}(SU(2)) $-comodule $\Omega$ together with a map $d:\cO_{q,s}(S^2) \to \Omega$ such that, for all $x,y\in \cO_{q,s}(S^2) $,
$$
\begin{aligned}
    & \mathrm{(a)}\qquad d(xy)=d(x)y+xd(y), \\
    & \mathrm{(b)}\qquad \Omega = \{\sum_ix_id(y_i)\;|\; x_i,y_i\in \cO_{q,s}(S^2) \},\\
    &\mathrm{(c)}\qquad \varrho\circ d = (\id \otimes d)\circ \Delta,
\end{aligned}
$$
where $\varrho: \Omega \to \cO_{q}(SU(2)) \ot \Omega$ is the left coaction and $\Delta$ is the comultiplication of $\cO_{q}(SU(2)) $ (remember that $\Delta (\cO_{q,s}(S^2) )\subset \cO_q(SL(2))\otimes\cO_{q,s}(S^2) $). 

As proven in \cite[Section~6]{BrzMaj:fac} there is a quantum version of the classical Hopf fibration (i.e.\ the quantum Podle\'s sphere is a base for a quantum principal bundle with the fibre $U(1)$ and total space the quantum group $SU(2)$). In algebraic terms this means that $\cO_{q}(SU(2)) $ is a principal coalgebra extension of $\cO_{q,s}(S^2) $ by a coalgebra $C$ spanned by grouplike elements $c_{n}$, $n\in \ZZ$. The coalgebra $C$ is defined as a quotient of $\cO_{q}(SU(2)) $ by the coideal and right $\cO_{q}(SU(2)) $-module $J= \langle \zeta, \xi+s,\eta +s\rangle \cO_{q}(SU(2)) $. By \cite[Theorem~2.2]{MulSch:hom}, $C$ is spanned by grouplike elements; their exact form is given in \cite[Section~6]{BrzMaj:fac}. The canonical surjection $\pi: \cO_{q}(SU(2)) \to C$ induces the right coaction $(\id\otimes\pi)\circ \Delta : \cO_{q}(SU(2)) \to \cO_{q}(SU(2)) \otimes C$. The algebra $\cO_{q,s}(S^2) \subset \cO_{q}(SU(2)) $ is then recovered as the coinvariants of this coaction, that is,
$$
\cO_{q,s}(S^2)  = \{x\in \cO_{q}(SU(2)) \; |\; (\id\otimes\pi)\circ \Delta (x) = x\otimes \pi(1)\}.
$$
Furthermore, the canonical Galois map
$$
\begin{aligned}
    \can : \cO_{q}(SU(2)) \otimes_{\cO_{q,s}(S^2) }\cO_{q}(SU(2)) &\to \cO_{q}(SU(2)) \otimes C,\\ 
    x\otimes y &\mapsto\sum xy\sw 1\otimes \pi(y\sw 2),
\end{aligned}
$$
is bijective (the quantum freeness condition). Moreover, the multiplication map $\cO_{q,s}(S^2) \otimes\cO_{q}(SU(2)) \to \cO_{q}(SU(2)) $ has a left $\cO_{q,s}(S^2) $-module right $C$-comodule section $\sigma$ (splitting; the equivariant projectivity).

The existence of the splitting $\sigma$ is equivalent to the existence of the \textit{strong connection lifting} $\ell: C\to \cO_{q}(SU(2)) \ot \cO_{q}(SU(2)) $ (for the definition and relation of $\ell$ to $\sigma$ the reader is referred to \cite{BrzHaj:Gal}). This is given iteratively as follows. Let $\ell_n := \ell(c_n)$. Then $\ell_0 = 1\otimes 1$ and, for all $n\in \NN$,
\begin{subequations}\label{str.con}
    \begin{equation}\label{str.con+}
    \ell_{n+1}=\frac{(q\gamma+q^{-n}s\delta)
\ell_{n}(-\beta+q^{-n}s\alpha)
+(\alpha +
q^{-n-1}s\beta)\ell_{n}(\delta-q^{-n}s\gamma)}{1+q^{-2n}s^2},
\end{equation}
\begin{equation}\label{str.con-}
    \ell_{-n-1}=\frac{(\delta- q^{n+1} s\gamma)
\ell_{-n}(\alpha+ q^{n}
s\beta)
+(\alpha q^{n}s-
q^{-1}\beta)\ell_{-n}(q^{n}s\delta+\gamma)}{1+q^{2n}s^2}.
\end{equation}
\end{subequations}

Let us write $\ell_n = \sum_{i\in I} \ell'_{n,i}\otimes \ell''_{n,i}$. As argued in \cite[Section~6]{BrzMaj:fac}, for all $n$, $\sum_{i\in I} \ell'_{n,i} \ell''_{n,i} = 1$. Furthermore, following  \cite[Theorem~3.1]{BrzHaj:Che}, for all $n,i,j$, $\ell''_{n,i}\ell'_{n,j} \in \cO_{q,s}(S^2) $. Then, by a straightforward argument $\{\ell''_{n,i}\}_{i\in I}$ generate projective left $\cO_{q,s}(S^2) $-modules $\cE_n$, while $\{\ell'_{-n,i}\}_{i\in I}$ generate projective right $\cO_{q,s}(S^2) $-modules $\tilde\cE_{n}$. Each of these modules can be interpreted  as a line bundle (viewed as a left or right module, correspondingly) of  
the topological charge $n\in \ZZ$. 
Using relations \eqref{str.con} one finds that each of the modules $\cE_{2}$ and $\tilde{\cE}_{-2}$ is generated by three elements. Explicitly, $\cE_{2}$ is generated by:
\begin{subequations}\label{gen.2}
    \begin{equation}
    e^+_{1} = \beta^2 - s(1+q^{-2})\alpha\beta + q^{-1}s^2\alpha^2, 
    \end{equation}
    \begin{equation}
        e^+_2 = q\beta\delta - qs(1+(q+q^{-1})\beta\gamma) + s^2\alpha\gamma,
    \end{equation}
    \begin{equation}
        e^+_3 = -q\delta^2 + s(q+q^{-1})\gamma\delta -s^2\gamma^2.
\end{equation}
\end{subequations}
Similarly, $\tilde{\cE}_{-2}$ is generated by:
\begin{subequations}\label{gen.-2}
    \begin{equation}
    e^-_{1} = q^{2}\gamma^2 + s(q+q^{-1})\gamma\delta + q^{-1}s^2\delta^2, 
    \end{equation}
    \begin{equation}
        e^-_2 = -q^2\alpha\gamma - qs(1+(q+q^{-1})\beta\gamma) - q^{-1}s^2\beta\delta,
    \end{equation}
    \begin{equation}
        e^-_3 = -q\alpha^2 - s(1+q^{-2})\alpha\beta -q^{-2}s^2\beta^2.
\end{equation}
 \end{subequations}

 \begin{lemma}\label{lem.lift}
   For any $a,c\in \CC$, let
   \begin{equation}\label{es}
       E^s_{a,c} = aE +\frac{s(a-c)}{q-q^{-1}}(K-1) +s^2qcKF \in U_q(sl_2).
   \end{equation}
   Then 
   \begin{equation}\label{skew.e}
       \Delta(E^s_{a,c})= E^s_{a,c}\otimes K +1\otimes E^s_{a,c},
   \end{equation}
   and
   \begin{equation}\label{lifting}
       E^s_{a,c}\la \xi = (a-cs^2) e^+_1, \quad E^s_{a,c}\la \zeta = (a-cs^2) e^+_2, \quad E^s_{a,c}\la \eta = (a-cs^2) e^+_3 .
   \end{equation}
 \end{lemma}
 \begin{proof}
     As explained in \cite{NouMim:Ask} the space of skew-primitive elements in $U_q(sl_2)$ of the type as in \eqref{skew.e} is three-dimensional, spanned by $E$, $K-1$ and $KF$. Since $E^s_{a,c}$ is in this space it is skew-primitive (this can also be checked directly). Equation \eqref{skew.e} facilitates calculations of the action of $E^s_{a,c}$ on the elements of $\cO_q(SL(2))$, as, for all $x,y\in \cO_q(SL(2))$,
     $$
     E^s_{a,c}\la (xy) = (E^s_{a,c}\la x)K\la y + x (E^s_{a,c}\la y).
     $$
     In view of equations \eqref{action} one easily finds:
     $$
     \begin{aligned}
         E^s_{a,c}\la \alpha &= a\beta - \frac{s(a-c)}{1+q}\alpha, \qquad &E^s_{a,c}\la \beta = \frac{s(a-c)}{1+q}q\beta + s^2 c \alpha\\
         E^s_{a,c}\la \gamma &= a\delta - \frac{s(a-c)}{1+q}\gamma, \qquad &E^s_{a,c}\la \delta = \frac{s(a-c)}{1+q}q\delta + s^2 c \gamma
     \end{aligned}
     $$
     Consequently,
     $$
     \begin{aligned}
         E^s_{a,c}\la \alpha^2 &= (1+q^{-2})a\alpha\beta - q^{-1}s(a-c)\alpha^2 , \quad E^s_{a,c}\la (\alpha\beta) = qa\beta^2 +s^2c\alpha^2\\
E^s_{a,c}\la \beta^2 &= s^2(q+q^{-1})c\alpha\beta + qs(a-c)\beta^2.
  \end{aligned}
     $$
   Combining these expressions in the way that defines $\xi$ one easily verifies the first of equalities \eqref{lifting}.  Next,
$$
\begin{aligned}
         E^s_{a,c}\la (\beta\gamma) &= a\beta\delta +s^2q^{-1}c\alpha\gamma , \quad E^s_{a,c}\la (\beta\delta) = s(a-c)q\beta\delta +s^2qc + s^2(1+q^2)c\beta\gamma \\
 E^s_{a,c}\la (\alpha\gamma) &= s(c-a)q^{-1}\alpha\gamma +a + (q+q^{-1})a\beta\gamma.
  \end{aligned}
     $$
Hence, $E^s_{a,c}\la\zeta = (a-cs^2)e^+_2$. Finally,
$$
     \begin{aligned}
         E^s_{a,c}\la \gamma^2 &= (1+q^{-2})a\gamma\delta - q^{-1}s(a-c)\gamma^2 , \quad E^s_{a,c}\la (\gamma\delta) = qa\delta^2 +s^2c\gamma^2\\
E^s_{a,c}\la \delta^2 &= s^2(q+q^{-1})c\gamma\delta + qs(a-c)\delta^2,
  \end{aligned}
     $$
     so that $E^s_{a,c}\la\eta = (a-cs^2)e^+_3$.
 \end{proof}

Similarly, one can prove the following observations.
\begin{lemma}\label{lem.lift-}
   For all $\tilde{a},\tilde{c}\in \CC$, let
   \begin{equation}\label{fs}
       F^s_{\tilde{a},\tilde{c}} = -q\tilde{a}F +\frac{s(\tilde{a}-\tilde{c})}{q-q^{-1}}(K^{-1}-1) -s^2\tilde{c}K^{-1}E \in U_q(sl_2).
   \end{equation}
   Then 
   \begin{equation}\label{skew.f}
       \Delta(F^s_{\tilde{a},\tilde{c}})= F^s_{\tilde{a},\tilde{c}}\otimes 1 +K^{-1}\otimes F^s_{\tilde{a},\tilde{c}},
   \end{equation}
   and
   \begin{equation}\label{lifting-}
       F^s_{\tilde{a},\tilde{c}}\la \eta = (\tilde{a}-\tilde{c}s^2) e^-_1,  \quad F^s_{\tilde{a},\tilde{c}}\la \zeta = (\tilde{a}-\tilde{c}s^2) e^-_2, \quad 
       F^s_{\tilde{a},\tilde{c}}\la \xi = (\tilde{a}-\tilde{c}s^2) e^-_3.
   \end{equation}
 \end{lemma}

 With these preliminary observations at hand we can state and prove the main result of the present note.
 
 \begin{proposition}\label{prop.calc}
     \begin{zlist}
         \item Let $\sigma: \cO_{q,s}(S^2)  \to \cO_q(SL(2))$ be the algebra mo\-no\-morphism given by the action of $K$, that is, $\sigma:x\mapsto K\la x$. Let  
     $\Omega^{(0,1)} = \cE_2\sigma(\cO_{q,s}(S^2) )$, viewed as an $\cO_{q,s}(S^2) $-bimodule by the left action $x\cdot \omega = x\omega$ and the right action
     $\omega \cdot x := \omega \sigma(x)$, for all $\omega \in \Omega^{(0,1)}$ and $x\in \cO_{q,s}(S^2) $. Then, for all $a,c\in \CC$ such that $a- s^2c=1$,  $\Omega^{(0,1)}$, together with  
     $$
     \bar{\partial}: \cO_{q,s}(S^2)  \to \Omega^{(0,1)}, \quad x\mapsto E^s_{a,c}\la x,
     $$ 
     is  a  covariant first-order differential calculus on $\cO_{q,s}(S^2) $.
     \item 
     Let $\tilde{\sigma}: \cO_{q,s}(S^2)  \to \cO_q(SL(2))$ be the algebra monomorphism given by the action of $K^{-1}$, that is, $\tilde{\sigma}:x\mapsto K^{-1}\la x$. Let  
     $\Omega^{(1,0)} = \tilde{\sigma}(\cO_{q,s}(S^2) )\tilde{\cE}_{-2}$, viewed as an $\cO_{q,s}(S^2) $-bimodule by the left action $x\cdot \omega = \tilde{\sigma}(x)\omega$ and the right action
     $\omega \cdot x := \omega x$, for all $\omega \in \Omega^{(1,0)}$ and $x\in \cO_{q,s}(S^2) $. Then, for all $\tilde a,\tilde c$ such that $\tilde a - s^2\tilde c =1$, $\Omega^{(1,0)}$ is a first-order covariant differential calculus on $\cO_{q,s}(S^2) $ with the derivation
     $$
     {\partial}: \cO_{q,s}(S^2)  \to \Omega^{(1,0)}, \quad x\mapsto F^s_{\tilde{a},\tilde{c}}\la x.
     $$ 
     \item 
         Furthermore, $
     {\Omega^{(1,0)}}^* = \Omega^{(0,1)},
     $
     and, for all $x\in \cO_{q,s}(S^2) $,
     $$
     \bar{\partial}(x^*) = \partial(x)^*.
     $$
     \end{zlist}
\end{proposition}
 \begin{proof}
 (1) By Lemma~\ref{lem.lift}, for all $x,y\in \cO_{q,s}(S^2) $,
 $$
 \bpartial(xy) = \bpartial(x)\sigma(y) + x\bpartial(y)= \bpartial(x)\cdot y + x\cdot \bpartial(y),
 $$
 hence, $\bpartial$ is an $\Omega^{(0,1)}$-valued derivation on $\cO_{q,s}(S^2) $. Finally, any element of $\Omega^{(0,1)}$ is a linear combination of the $x e^+_i \sigma(y)$, with $x,y\in \cO_{q,s}(S^2) $. By the Leibniz rule,
 $$
 x e^+_i \sigma(y) = x\bpartial(\zeta_iy) - x\zeta_i\bpartial(y), 
 $$
where $\zeta_i$  is the generator of $\cO_{q,s}(S^2) $ such that $e^+_i = \bpartial(\zeta_i)$, which exists by \eqref{lifting}. Therefore, $\Omega^{(0,1)} = \cO_{q,s}(S^2) \bpartial\left(\cO_{q,s}(S^2) \right)$, as required for a first-order differential calculus over $\cO_{q,s}(S^2) $.

Since $\bar\partial$ is given by the action \cref{act.gen} of $U_q(su_2)$ on $\cO_{q,s}(S^2) $, 
\begin{equation}\label{part.cov}
  (\id \otimes \bar\partial)\circ \Delta = \Delta \circ \bar\partial,  
\end{equation}
 hence $\bpartial\left(\cO_{q,s}(S^2) \right)$ is a left $\cO_{q}(SU(2)) $-coideal.
 Furthermore, since $\Omega^{(0,1)} = \cO_{q,s}(S^2) \bpartial\left(\cO_{q,s}(S^2) \right)$, we have that $\Omega^{(0,1)}$ is a left $\cO_{q}(SU(2)) $-subcomodule and hence also an $\cO_{q}(SU(2))$-comodule,
while \cref{part.cov} states that $\bpartial$ is a comodule map. Put together all this means that $\Omega^{(0,1)}$ is a covariant calculus.

Statement (2) is proven in a similar way using Lemma~\ref{lem.lift-}.

(3) This follows from the observation that ${e_i^+}^* =e_i^-$, and in view of  \eqref{*comp}, for all $x\in \cO_{q,s}(S^2) $ and $y\in \cO_{q}(SU(2)) $,
$$
\tilde{\sigma}(x) = \sigma(x^*)^*, \quad \langle F^s_{\bar a,\bar c}, y^*\rangle  = \overline{\langle E^s_{a,c}, y\rangle}.
$$
\end{proof}

Note that the conditions for $a,c$ and $\tilde a, \tilde c$ can be replaced by $a\neq cs^2$ and $\tilde{a}\neq \tilde{c}s^2$ and by the appropriately rescaled derivations. As rescaling of a derivation by a non-zero number is an automorphism  of a differential calculus, the relaxing of the conditions does not lead to any new calculi.

\begin{corollary}\label{cor.calc}
With the assumptions and notation as in Proposition~\ref{prop.calc}, the pair
$$
d=\partial + \bpartial, \qquad \Omega^1 = \span\{x dy\;|\;x,y\in \cO_{q,s}(S^2) \} \subseteq \Omega^{(1,0)} \oplus \Omega^{(0,1)},
$$
 is a  first-order $\cO_{q}(SU(2)) $-covariant  $*$-differential calculus on $\cO_{q,s}(S^2) $. 
\end{corollary}

\end{document}